\documentclass[a4paper,12pt]{amsart}
\newtheorem{thm}{Theorem}%[section]
\newtheorem*{thm*}{Theorem}

\newtheorem{cor}[thm]{Corollary}
\newtheorem{lem}[thm]{Lemma}
\theoremstyle{definition}

\newtheorem*{defn*}{Definition}

\usepackage{amssymb}
\usepackage{amsmath}
\usepackage{hyperref}

\newcommand{\R}{\mathbb R}
\newcommand{\N}{\mathbb N}

\def\ce{Ces\`{a}ro }
\def\ces{\mathcal{C}}
\def\linf{L^\infty}

\def\elp{L^p([0,1])}

\def\im{I_{m}}
\def\lm{{L^1(m)}}

\begin{document}

\title[The weak Banach-Saks property]
{The weak Banach-Saks   property \\ for function spaces}

\author{Guillermo P. Curbera}
\address{Facultad de Matem\'aticas \& IMUS,  Universidad de Sevilla,
Aptdo.  1160,  Sevilla 41080, Spain}
\email{curbera@us.es}

\author{Werner J. Ricker}
\address{Math.--Geogr. Fakult\"at, Katholische Universit\"at
Eichst\"att--Ingolstadt, D--85072 Eichst\"att, Germany}
\email{werner.ricker@ku.de}

\thanks{The first author acknowledges the support  of the
\lq\lq International Visiting Professor Program 2015\rq\rq\  from the
Ministry of Education, Science and Art, Bavaria (Germany), and of
MTM2012-36732-C03-03, MINECO (Spain).}

\date{\today}



\begin{abstract}
We establish the weak Banach-Saks property for function spaces
arising as the optimal domain of an  operator.
\end{abstract}

\maketitle

%%%%%%%%%%%%%%%%%%%%%%%%%%%%%%%%%%%%%%%%%%%%%%%%%%%%%%%%%%%%%%%%%%

\section*{Introduction}

%%%%%%%%%%%%%%%%%%%%%%%%%%%%%%%%%%%%%%%%%%%%%%%%%%%%%%%%%%%%%%%%%%%%

Astashkin and Maligranda  proved that the Banach function space (B.f.s.)
$$
Ces_p[0,1]:=\left\{f: x\mapsto \frac{1}{x}\int_0^x|f(y)|\,dy\in
L^p([0,1])\right\} ,\quad 1\le p<\infty,
$$
has the \textit{weak Banach-Saks property}, namely, every
weakly null sequence in $Ces_p[0,1]$ admits a subsequence whose arithmetic means converge
to zero in the norm of $Ces_p[0,1]$, \cite[\S7]{astashkin-maligranda1}.
The space $\elp$, for $1\le p<\infty$, itself has the weak Banach-Saks property.
This is due to Banach and Saks for $1<p<\infty$, \cite{banach-saks},
and to Szlenk for $p=1$, \cite{szlenk}. An important step in the proof  of the
above result in \cite{astashkin-maligranda1} is to first establish that $Ces_p[0,1]$ satisfies
the  \textit{subsequence splitting property}.
This property goes back to a celebrated paper of Kadec and Pe\l czy\'nski, \cite{kadec-pelczynski}, where
they observed  that in $\elp$, $1\le p<\infty$, every norm bounded sequence $\{f_n\}$
has a subsequence $\{f'_{n}\}$ that can be split in the form
$f'_{n}=g_n+h_n$,
where the functions $\{h_n\}$ have pairwise disjoint support and the sequence
$\{g_n\}$ has uniformly absolutely continuous (a.c.)\ norm in $\elp$, that is,
$\sup_n\|g_n\chi_A\|_p\to0$ when $\lambda(A)\to0$, where $\lambda$ is Lebesgue measure on $[0,1]$.
Characterizations of the subsequence splitting property
(in terms of ultraproducts) are due to Weis, \cite{weis}; they play a crucial role in
Section \ref{S-ssp}.

The above results raise the question of whether the
subsequence splitting property and the weak Banach-Saks property are also satisfied
in analogous  B.f.s.',  such as, for example,
\begin{equation}\label{T}
\Big\{f\in L^1(G): \nu*|f|\in L^p(G)\Big\},\quad 1< p<\infty,
\end{equation}
where $\nu$ is a positive, finite Borel measure on a compact abelian group $G$;
or for
\begin{equation}\label{I}
\bigg\{f\colon[0,1]\to\R: I_\alpha(f)(x):=\int_0^1\frac{|f(y)|}{|x-y|^{1-\alpha}}\,dy\in X\bigg\},
\end{equation}
where $0<\alpha<1$ and $X$  is a rearrangement invariant (r.i.) space on $[0,1]$; or for
\begin{equation}\label{S}
\bigg\{f\colon[0,1]\to\R: T(f)(x):=\int_x^1 y^{(1/n)-1}|f(y)|\,dy\in X\bigg\},
\end{equation}
where $n\ge2$ and $X$  is a r.i.\  space on $[0,1]$.

The  common feature for  these types of B.f.s.'\   is that each one is the \textit{optimal extension domain}
of an appropriate  linear operator. Indeed, in the case of $Ces_p[0,1]$ this is so for the \ce operator
$$
f\mapsto \ces(f): x\mapsto \frac{1}{x}\int_0^x f(y)\,dy;
$$
see \cite{astashkin-maligranda1}; in the case of the B.f.s.\  in \eqref{T} for the operator of
convolution with the measure $\nu$, that is, for
$$
f\mapsto T_\nu(f)=f*\nu:x\mapsto\int_G f(y-x)\,d\nu(y);
$$
see \cite{okada-ricker2}; in the case of the B.f.s.\ in \eqref{I} for the
Riemann-Liouville fractional integral of order $\alpha$, that is, for
$$
f\mapsto I_\alpha(f):x\mapsto\int_0^1\frac{f(y)}{|x-y|^{1-\alpha}}\,dy ;
$$
see \cite{curbera-ricker1}; and in the case of the B.f.s.\ in \eqref{S} for
the kernel operator associated to the $n$-dimensional Sobolev
inequality, that is, for
$$
f\mapsto T_n(f):x\mapsto\int_x^1y^{(1/n)-1} f(y)\,dy ;
$$
see \cite{curbera-ricker2}, \cite{edmunds-kerman-pick}.

Concerning the optimal domain of an operator, consider
a finite measure space  $(\Omega,\Sigma,\mu)$, a
Banach space $X$ and  an  $X$-valued linear map $T$ defined on a
vector subspace $\mathcal{D}\subseteq L^0(\mu)$ which contains
$L^\infty(\mu)$. Here $L^0(\mu)$ is the space of classes of all a.e.\
$\R$-valued, measurable functions defined on $\Omega$.
Then the optimal domain for $T$, taking its values in $X$,
is  the linear space  defined by
\begin{equation*}
[T,X]:=\Big\{f\in L^0(\mu): T(|f|)\in X\Big\},
\end{equation*}
which becomes a B.f.s.\ when endowed with the norm
$$
\|f\|_{[T,X]}:=\|T(|f|)\|_X,\quad f\in[T,X].
$$
In this notation,
we have $Ces_p=[\ces,L^p([0,1])]$, the B.f.s.\  in \eqref{T} is
$[T_\nu,L^p(G)]$, the B.f.s.\  in \eqref{I} is $[I_\alpha,X]$,
and the B.f.s.\  in \eqref{S} is $[T_n,X]$.

The aim of this paper is to extend the above mentioned results  of Astashkin
and Maligranda to the setting of operators other than the \ce operator and B.f.s.'\
other than    $\elp$. For this we need to
determine  conditions on the Banach space
$X$ and on the operator $T$ which guarantee that the space $[T,X]$
has the subsequence splitting property
and the weak Banach-Saks property. This is achieved
in  Theorems  \ref{theorem4}, \ref{theorem5} and \ref{theorem6}. The combination
of these theorems leads to the following result.

Recall that a linear operator between Banach B.f.s.'\ is said to be \textit{positive} if it maps
positive functions to positive functions, in which case it is automatically continuous.

\begin{thm}\label{theorem1}
Let $(\Omega,\Sigma,\mu)$ be a separable, finite measure space,
$X$ be   a B.f.s.\  which  possesses  the weak Banach-Saks property and
such that both $X$ and $X^*$ have the subsequence splitting property. Let
$T\colon L^\infty(\mu)\to X$ be a positive, linear operator.
Then, the B.f.s.\  $[T,X]$ has both the subsequence splitting property and the weak Banach-Saks property.
\end{thm}

Given a measurable function $K\colon(x,y)\in[0,1]\times[0,1]\mapsto K(x,y)\in[0,\infty]$,
recall that the associated kernel operator $T_K$ is  defined  by
\begin{equation}\label{tk}
T_Kf(x):=\int_0^1K(x,y)f(y)\,dy,\quad x\in[0,1],
\end{equation}
for any measurable function $f$  for which it is meaningful to do so.

As a consequence of Theorem \ref{theorem1} we have the following result.

\begin{cor}\label{corollary2}
Let $X$ be a  r.i.\ space on $[0,1]$   which  possesses  the weak Banach-Saks property and
such that both $X$ and $X^*$ have the subsequence splitting property.
Let $K\colon(x,y)\in[0,1]\times[0,1]\mapsto K(x,y)\in[0,\infty]$
be a measurable kernel such that $T_K(\chi_{[0,1]})\in X$, where
$T_K$ is as in \eqref{tk}.
Then,  the B.f.s.\  $[T_K,X]$ has both the subsequence splitting property and
the Banach-Saks property.

In particular, the result holds whenever $X$ is reflexive and
possesses  the weak Banach-Saks property.
\end{cor}

From Corollary \ref{corollary2} it follows, for example,
that the B.f.s\  $[I_\alpha,X]$ in \eqref{I} corresponding to
the kernel $K(x,y):=|x-y|^{\alpha-1}$, the B.f.s.\ $[T_n,X]$ in \eqref{S} generated
by the Sobolev kernel $K(x,y)=y^{(1/n)-1}\chi_{[x,1]}(y)$, and
the \ce space $[\ces,X]$ corresponding to the kernel $K(x,y):=(1/x)\chi_{[0,x]}(y)$,
all have the subsequence splitting property and
the Banach-Saks property, whenever $X$ is a  reflexive r.i.\ space with the weak Banach-Saks property.
We refer to  Section 4 for the details and further examples, also including convolution operators by measures.

A comment regarding the techniques is in order. There is a (somewhat
unexpected, although classical) tool available for treating optimal domains
in a  unified way: there always exists an  underlying vector measure associated with the operator together with
its corresponding $L^1$-space  consisting of all  the scalar functions which are integrable
 with respect to that vector measure (in the sense of Bartle, Dunford and Schwartz).
Accordingly, Theorems  \ref{theorem4} and \ref{theorem5} are formulated for  the subsequence splitting property
and the weak Banach-Saks property for $L^1$-spaces  of a general vector measure,  respectively. For instance,
in the case of the \ce operator, the associated $L^p([0,1])$-valued vector measure is given by
\begin{equation*}
m_{L^p}:A\mapsto m_{L^p}(A):=\mathcal{C}(\chi_A),
\quad A\subseteq [0,1] \text{ measurable}.
\end{equation*}
For this vector measure it turns out that  $Ces_p[0,1]=L^1(m_{L^p})$.

%%%%%%%%%%%%%%%%%%%%%%%%%%%%%%%%%%%%%%%%%%%%%%%%%%%%%%%%%%%%%%%%%%

\section{Preliminaries}\label{S-prelim}

%%%%%%%%%%%%%%%%%%%%%%%%%%%%%%%%%%%%%%%%%%%%%%%%%%%%%%%%%%%%%%%%%%%%

A \textit{Banach function space} (B.f.s.)\ $X$ over a measure space
$(\Omega,\Sigma,\mu)$  is a
Banach space  of classes of measurable functions on $\Omega$ satisfying
the ideal property, that is, $g\in X$ and $\|g\|_X\le\|f\|_X$
whenever $f\in X$ and $|g|\le|f|$ $\mu$--a.e. We denote by $X^+$
the cone in $X$ consisting in all $f\in X$ satisfying $f\ge0$ $\mu$-a.e.
The B.f.s.\ $X$ has \textit{absolutely continuous} (a.c.)\  norm
if $\lim_{\mu(A)\to0}\|f\chi_A\|_X=0$ for $f\in X$; here $\chi_A$
denotes the characteristic function of a set $A\in\Sigma$. An equivalent
condition is that order bounded, increasing sequences in $X$ are norm convergent.
A subset $K\subseteq X$
is said to have \textit{uniformly a.c.}\ norm if $\lim_{\mu(A)\to0}\sup_{f\in K}\|f\chi_A\|_X=0$.
Sets with uniform a.c.\ norm are also called almost order bounded sets or L-weakly
compact sets. In B.f.s.' with a.c.\ norm, all relatively compact sets have
uniform a.c.\ norm, and all sets with  uniform a.c.\ norm are relatively
weakly compact; see \cite[\S3.6]{meyer-nieberg}.

A \textit{rearrangement invariant} (r.i.) space $X$ on $[0,1]$ is a
B.f.s.\  on $[0,1]$ such that if $g^*\le f^*$ and $f\in X$,  then $g\in X$ and $\|g\|_X\le\|f\|_X$.
Here $f^*$ is the decreasing rearrangement of $f$, that is, the
right continuous inverse of its distribution function:
$\mu_f(\tau):=\mu(\{t\in [0,1]:\,|f(t)|>\tau\})$.
If a r.i.\ space $X$ has a.c.\ norm, then  the dual space $X^*$  is again  r.i.
A r.i.\ space $X$ always satisfies $\linf\subseteq X\subseteq L^1$.

We recall briefly the theory of integration of real functions with respect
to a vector measure, due to Bartle, Dunford and  Schwartz, \cite{bartle-dunford-schwartz}.
Let $(\Omega ,\Sigma)$ be a measurable space,
$X$ be a Banach space with dual space
$X^*$ and closed unit ball $B_{X^*}$, and $m\colon\Sigma\to X$ be a $\sigma$-additive vector measure.
The \textit{semivariation} of $m$ is defined by
$$
A\mapsto\|m\|(A) := \sup\{ |x^*m|(A) : x^*\in B_{X^*} \},\quad A\in\Sigma,
$$
where $|x^*m|$ is the variation measure of the scalar measure $x^*m:
A\mapsto\langle x^*,m(A)\rangle$ for $A\in\Sigma$.
A \textit{Rybakov control measure} for $m$ is a measure of the form $\eta=|x_0^*m|$
for a suitable element $x_0^*\in X^*$ such that $\eta(A)=0$ if and only if $\|m\|(A)=0$;
see \cite[Theorem IX.2.2]{diestel-uhl}.

A measurable function $f\colon\Omega\to\R$ is called \textit{$m$--scalarly
integrable} if $f\in L^1(|x^*m|)$, for
every $x^*\in X^*$. The function $f$ is
\textit{$m$--integrable}  if, in addition, for each $A\in\Sigma$ there exists a
vector in $X$ (denoted by $\int_Af\,dm$) such that
$\langle\int_Af\,dm,x^*\rangle=\int_Af\,dx^*m$, for every
$x^*\in X^*$.   The $m$--integrable functions
form a linear space in which
$$
\|f\|_{\lm} : =\sup\bigg\{\int_\Omega |f|\,d|x^*m| : x^*\in B_{X^*}\bigg\}
$$
is a seminorm. Identifying functions which differ  $\|m\|$--a.e.,
we obtain a Banach space (of classes) of $m$--integrable
functions, denoted by $\lm$. It is a B.f.s.\ over $(\Omega,\Sigma,\eta)$
relative to any Rybakov control measure $\eta$ for $m$.
Simple functions are dense in $\lm$, the $m$--essentially bounded functions are
contained in $\lm$, and $\lm$ has  a.c.\ norm.
This last property implies that $\lm^*$ can be identified with its associate space, that is,
with the space of all measurable functions $g$ satisfying
$fg\in L^1(\eta)$ for all $f\in\lm$; the identification
is given by $f\in\lm\mapsto\int_\Omega fg\,d\eta\in\R$.
In particular, $L^\infty\subseteq\lm^*$.
An equivalent norm for $\lm$ is given by
$|||f||| :=\sup\{\| \int_A f\, dm\| _X : A\in\Sigma\}$, which satisfies
$|||f|||\le ||f||_{\lm}\le 2 |||f|||$ for $f\in\lm$.
The integration operator $\im\colon\lm\to X$ is defined by $f\mapsto\int_\Omega f\,dm$. It is
continuous, linear and has operator norm  one. It should be noted 
that the spaces $\lm$ can be quite different to the
classical $L^1$--spaces of scalar measures.  Indeed, every Banach
lattice with a.c.\ norm and having a weak unit (e.g., $L^2([0,1])$)
is the $L^1$--space of some vector measure, \cite[Theorem 8]{curbera1}.

For further details concerning
B.f.s.'\  and r.i.\ spaces we refer to
\cite{lindenstrauss-tzafriri}. For further facts related to the spaces
$\lm$ see \cite{okada-ricker-sanchez}.

The following result is implicit in the construction
of the Bartle, Dunford, Schwartz integral (cf.\  the proof
of Theorem 2.6(a) of \cite{bartle-dunford-schwartz}), although it is  not
explicitly stated. We include a proof for the sake of completeness.

\begin{lem}\label{lewis}
Let $\{f_n\}\subseteq \lm$ be a sequence satisfying
\begin{itemize}
\item[(a)] $f_n(x)\to f(x)$ for $||m||$-a.e.\ $x\in\Omega$, and
\item[(b)] $\left\{\int_A f_n\,dm\right\}$ is convergent in $X$, for each $A\in\Sigma$.
\end{itemize}
Then, $f\in\lm$ and $\{f_n\}$ converges to $f$ in the norm of $\lm$.
\end{lem}

\begin{proof}
For each $f_n\in\lm$, $n\in\N$,  the set function
$A\mapsto \int_Af_n\,dm\in X$, $A\in\Sigma$, is a
$\sigma$-additive measure (due to the Orlicz-Pettis Theorem),
which is absolutely continuous with respect to a control
measure for $m$. This fact, together with (b) implies, via the
Vitali-Hahn-Saks Theorem, \cite[Theorem I.5.6]{diestel-uhl}, that the convergence
in (b) is uniform with respect to the sets $A\in\Sigma$. Accordingly,
$f\in\lm$, \cite[Theorem 2.7]{bartle-dunford-schwartz}.
The convergence $f_n\to f$ in norm in $\lm$ follows directly by considering the equivalent
norm $|||\cdot|||$ in $\lm$.
\end{proof}

%%%%%%%%%%%%%%%%%%%%%%%%%%%%%%%%%%%%%%%%%%%%%%%%%%%%%%%%%%%%%%%%%%

\section{The subsequence splitting property for $\lm$}\label{S-ssp}

%%%%%%%%%%%%%%%%%%%%%%%%%%%%%%%%%%%%%%%%%%%%%%%%%%%%%%%%%%%%%%%%%%%%

In \cite[2.1 Definition]{weis} Weis gives a general
definition of the subsequence splitting property. Let $X$ be a B.f.s.\
with a.c.\ norm defined over a measure space $(S,\sigma,\mu)$. Then $X$
has the \textit{subsequence splitting property}  if for every
norm bounded sequence $\{f_n\}\subseteq X$ there is a subsequence $\{f'_{n}\}$ of $\{f_n\}$
and sequences $\{g_n\}$ and $\{h_n\}$ in $X$ such that:
\begin{itemize}
\item[(a)] For $n\in\N$ we have $f'_{n}=g_n+h_n$, with $g_n$ and $h_n$ having disjoint support.
\item[(b)]  The sequence $\{g_n\}$ has uniformly a.c.\ norm in $X$.
\item[(c)] The functions $\{h_n\}$ have pairwise disjoint support.
\end{itemize}

Weis gives several characterizations of the subsequence splitting property, \cite[2.5 Theorem]{weis}.
We select only those which are required in the sequel. Namely,
\begin{itemize}
\item[(i)] $X$ has the subsequence splitting property,
\item[(ii)] $\tilde X$ has a.c.\ norm,
\item[(iii)] $\tilde X$ does not contain a copy of $c_0$,
\end{itemize}
where the space $\tilde X$ is  constructed  as follows; see \cite{weis}.
Let $\mathcal{U}$ be a free ultrafilter in $\N$. Consider the ultraproduct
of $X$ via $\mathcal{U}$ given by the quotient
$$
X_{\mathcal{U}}:=\ell_\infty(X)\big/N_{\mathcal{U}},\quad\hbox{where}\quad
N_{\mathcal{U}}=\left\{\{f_n\}\in\ell_\infty(X) : \lim_{\mathcal{U}}\|
f_n\|_X=0\right\}
$$
and $\ell_\infty(X)$ is the space of all bounded sequences in $X$.
Denote by $[f_n]\in X_{\mathcal{U}}$ the equivalence class of the element
$\{f_n\}\in\ell_\infty(X)$. The space $X_{\mathcal{U}}$ becomes a Banach lattice for the following norm and
order:
\begin{equation*}
\|[f_n]\|_{\mathcal{U}}:=\lim_{\mathcal{U}}\|f_n\|_X,\quad \inf\{[f_n],[g_n]\}:=[\inf\{f_n,g_n\}].
\end{equation*}
For details on ultraproducts of Banach spaces see \cite{heinrich}.
Let $\chi_S$ be the characteristic function of the underlying set $S$. Then, $[\chi_S]$ is the
equivalence class of the constant sequence $\{\chi_S\}$. Let $\tilde X$ denote the
band  in $X_{\mathcal{U}}$  generated by  $[\chi_S]$, that is,
$\tilde X=[\chi_S]^{\perp\perp}$. Recall that a band $M$ in a Banach lattice $Z$ is a
closed subspace which is an order ideal (i.e., $f\in M$ and $g\in Z$ with $|g|\le |f|$ imply $g\in M$)
and  is closed under the formation of suprema, \cite[p.3]{lindenstrauss-tzafriri}.

%%%%%%%%%%%%%%%%%%%%%%%%%%%%%%%%%%%%%%%%%%%%%%%%%%%%%%%%%%%%%%%%%%%%%%%%%%%%%%%%%%%%%%

\medskip

\begin{thm}\label{theorem4}
Let $X$ be  a B.f.s.\
%with a.c.\ norm
and $m\colon\Sigma\to X$
be a $\sigma$-additive vector measure.
The following conditions are assumed to hold.
\begin{itemize}
\item[(a)] $X$ and $X^*$ have the subsequence splitting property.
\item[(b)] The range $m(\Sigma)$ of $m$ has uniformly a.c.\ norm in $X$.
\end{itemize}
Then, the B.f.s.\ $\lm$ has the subsequence splitting property.
\end{thm}

\begin{proof}
Recall, since $X$ has the subsequence splitting property, that it has a.c.\ norm.
In order to prove the result we  construct an $\tilde X$-valued $\sigma$-additive measure $\tilde{m}$
with the property that  $(\lm)\,\tilde{}\;$ is order isomorphically contained  in the
B.f.s.\  $L^1(\tilde{m})$. A  general result asserts that every $L^1$-space  of a vector measure
has a.c.\ norm, \cite[Theorem 1]{curbera1},
and hence, $(\lm)\,\tilde{}\;$ has a.c.\ norm. Then, by the characterization (ii) recorded above,
it follows that $\lm$ has the subsequence splitting property.

Let $\eta$ be a Rybakov control measure for  $m$. Then, with continuous inclusions, we have
\begin{equation}\label{bfs}
L^\infty(\Omega,\Sigma,\eta)\subseteq\lm\subseteq L^1(\Omega,\Sigma,\eta).
\end{equation}
Fix a free ultrafilter $\mathcal{U}$  in $\N$.
Then the  ultraproduct of  $L^1(\Omega,\Sigma,\eta)$ via $\mathcal{U}$ can be identified as
$$
L^1(\Omega,\Sigma,\eta)_{\mathcal{U}}
= L^1(\tilde\Omega,\tilde\Sigma,\tilde\eta)\oplus\Delta' ,
$$
where $(\tilde\Omega,\tilde\Sigma,\tilde\eta)$ is a measure space and
the elements of $\Delta'$ are disjoint from $[\chi_\Omega]$; see  \cite[\S4]{dacunha-castelle-krivine1},
\cite{dacunha-castelle-krivine2}, \cite[\S3]{gonzalez-abejon}.
Thus, it follows that
$$
(L^1(\Omega,\Sigma,\eta))\,\tilde{}=  L^1(\tilde\Omega,\tilde\Sigma,\tilde\eta).
$$
The same procedure can be done with $L^\infty(\Omega,\Sigma,\eta)$.
This allows the identification of  $(\lm)\,\tilde{}\;$ with a function space by forming the
ultraproducts of the inclusions in \eqref{bfs}, namely
$$
L^\infty(\tilde\Omega,\tilde\Sigma,\tilde\eta)\subseteq
(\lm)\,\tilde{}\subseteq L^1(\tilde\Omega,\tilde\Sigma,\tilde\eta),
$$
with both inclusions being continuous.

The $\sigma$--algebra $\tilde\Sigma$ is isomorphic to the Boolean ring
$\{[\chi_{A_n}]:A_n\in\Sigma\}$ formed in the quotient space
$L^1(\Omega,\Sigma,\lambda)_{\mathcal{U}}$.
Thus, every measurable set $\tilde A\in\tilde\Sigma$ can be identified
with a sequence of sets $\{A_n\}$  with each $A_n\in\Sigma$, where two sequences
of measurable sets $\{A_n\}$ and $\{B_n\}$ are identified if $\lim_{\mathcal{U}} \eta(A_n\bigtriangleup
B_n)=0$. Here $A\triangle B$ denotes
the symmetric difference of two sets $A$ and $B$. The measure $\tilde\eta$ is then defined via
$$
\tilde A=\{A_n\}\in\tilde\Sigma\mapsto\tilde\eta(\tilde A):=\lim_{\mathcal{U}}
\eta(A_n)\in\R^+.
$$
A function $\tilde f$ in $L^1(\tilde\Omega,\tilde\Sigma,\tilde\eta)$
is an element $[f_n]$ in $L^1(\Omega,\Sigma,\eta)_{\mathcal{U}}$,
and the integral of $\tilde f$ over measurable sets with respect to $\tilde\eta$ is defined as
$$
\int_{\tilde A} \tilde f\,d\tilde\eta=\lim_{\mathcal{U}} \int_{A_n}
f_n\,d\eta,\quad \tilde A=\{A_n\}\in\tilde\Sigma.
$$
For further details, see \cite[\S5]{heinrich}.

We define a vector measure $\tilde{m}$ by
$$
\tilde A=\{A_n\}\in\tilde\Sigma\longmapsto\tilde{m}(\tilde
A)=[m(A_n)]\in X_{\mathcal{U}}.
$$
As $m(\Sigma)$ is a bounded subset of $X$ it is clear that
 $\tilde{m}$ is well defined. Moreover, $\tilde m$ is finitely additive,
\cite[p.322]{dacunha-castelle-krivine1}. To verify its $\sigma$-additivity, let $\varepsilon>0$.
As $m$ is absolutely
continuous with respect to $\eta$, there exists $\delta>0$ such that
if $\eta(A)<\delta$, then $\|m\|(A)<\varepsilon$. Let $\tilde
A=\{A_n\}\in\tilde\Sigma$ satisfy $\tilde\eta(\tilde A)<\delta$, that
is, $\lim_{\mathcal{U}}\eta(A_n)<\delta$. Then there exists $V\in \mathcal{
U}$ such that for every $n\in V$ we have $\eta(A_n)<\delta$. Thus,
for every $n\in V$ it follows that $\|m(A_n)\|\le\|m\|(A_n)<\varepsilon$.
So, $\|\tilde{m}(\tilde A)\|_{\mathcal{U}}<\varepsilon$.
Hence, $\tilde{m}$ is absolutely continuous with respect to $\tilde\eta$
from which we deduce that $\tilde{m}$ is $\sigma$-additive.

By hypothesis,  the range $m(\Sigma)$ of the measure $m$ has uniformly a.c.\ norm
in $X$. In order to show that  the measure $\tilde{m}$ actually
takes its values in $\tilde X\subseteq X_{\mathcal{U}}$
we use  \cite[1.5 Proposition]{weis} which asserts that if $\{f_n\}$
has uniformly a.c.\  norm in $X$, then $[f_n]\in \tilde X$. Let
$\tilde A=\{A_n\}\in\tilde\Sigma$. Then $\tilde{m}(\tilde A)=[m(A_n)]\in X_{\mathcal{U}}$.
But, $\{m(A_n)\}\subseteq m(\Sigma)$ which has uniformly a.c.\ norm. Hence,
$\tilde{m}(\tilde A)=[m(A_n)]\in \tilde X$.

Next, we prove that $(\lm)\,\tilde{}\;$ is contained in $L^1(\tilde{m})$.
To this aim, it suffices to show that each $\tilde f\in(\lm)\,\tilde{}\;$
is scalarly $\tilde{m}$-integrable. The reason for
this is two-fold. On the one hand, $\tilde X$ does not contain a copy of $c_0$ since
$X$ satisfies the subsequence splitting
property; see (iii) above. On the other hand,
for vector measures  with values in a Banach space not containing $c_0$,
integrability and scalar integrability coincide, \cite[Theorem 5.1]{lewis2}.

Since $X$ and $X^*$ satisfy the subsequence splitting
property, we have $(\tilde X)^*=(X^*)\,\tilde{}$ and the norms in both spaces
coincide, \cite[Corollary 2.7]{weis}. Hence, the elements of $(\tilde X)^*$ are
of the form $\tilde g^*=[g_n^*]$ for $\{g_n^*\}$ a
bounded sequence in $X^*$.

Fix $\tilde g^*\in(\tilde X)^*$. The scalar measure $\tilde g^*\tilde{m}\colon\tilde\Sigma\to\R$
is absolutely continuous with
respect to $\tilde\eta$ (since  $\tilde{m}$ is absolutely continuous with
respect to $\tilde\eta$). Thus, $\tilde g^*\tilde{m}$  has a
Radon--Nikodym derivative with respect to $\tilde\eta$. We denote it by
$h_{\tilde g^*}$; it belongs to  $L^1(\tilde\eta)$.

Let $\tilde A=\{A_n\}\in\tilde\Sigma$. Then,
\begin{eqnarray*}
\big\langle \tilde g^*,\tilde{m}(\tilde A)\big\rangle
   &=& \big\langle  [g_n^*],[m(A_n)]\big\rangle
= \lim_{\mathcal{U}}\big\langle g_n^*,m(A_n)\big\rangle
\\ &=& \lim_{\mathcal{U}}\int_{A_n}1\,d(g_n^*m) =\lim_{\mathcal{U}}\int_{A_n} h_{g_n^*}\,d\eta
= \int_{\tilde A} \tilde h\,d\tilde\eta,
\end{eqnarray*}
where $\tilde h:=[h_{g_n^*}]$ and $h_{g_n^*}\in L^1(\eta)$ is the Radon--Nikodym
derivative of the measure $g_n^*m$ with respect to $\eta$, for each $n\in\N$.
Hence, $h_{\tilde g^*}=[h_{g_n^*}]$.

Let now $\tilde f\in(\lm)\,\tilde{}\;$. Then  $\tilde f=[f_n]$ for  a bounded sequence $\{f_n\}$ in $\lm$, with
$\|\tilde f\|_{\mathcal{U}}=\lim_{\mathcal{U}}\|f_n\|_{\lm}$.
Accordingly,
\begin{eqnarray*}
\int |\tilde f|\, d|\tilde g^*\tilde{m}| &=& \int |\tilde f|\cdot|\tilde h_{\tilde g^*}|\, d\,\tilde\eta
= \lim_{\mathcal{U}} \int |f_n|\cdot| h_{g_n^*}|\,d\eta
\\ &=& \lim_{\mathcal{U}} \int |f_n|\, d|g_n^*m|
\le  \lim_{\mathcal{U}} \|f_n\|_{\lm}\cdot\|g_n^*\|_{X^*}
\\ &=& \|\tilde f\|_{\mathcal{U}}\cdot\|\tilde g^*\|_{\mathcal{U}} .
\end{eqnarray*}
Hence, $\tilde f$ is integrable with respect to $\tilde g^*\tilde{m}$.
It follows that $\tilde f$ is scalarly $\tilde{m}$-integrable and hence,
integrable with respect to the vector measure $\tilde{m}$. We also deduce
from the previous inequality that
$$
\|\tilde f\|_{L^1(\tilde{m})}\le \|\tilde f\|_{\mathcal{U}},\quad
\tilde f\in(\lm)\,\tilde{}\; .
$$

Let $\varepsilon>0$. By using the equivalent norm $|||\cdot|||$ in $\lm$,
we can select for every $n\in\N$, a measurable set $A_n$ such that
$$
\left\|\int_{A_n} f_n\,dm\right\|_X\ge \frac{1-\varepsilon}{2}\|f_n\|_{\lm}.
$$
Set $\tilde A:=\{A_n\}$ in $\tilde\Sigma$. Then
\begin{eqnarray*}
\|\tilde f\|_{L^1(\tilde{m})}  &\ge& \left\|\int_{\tilde A}\tilde f d\tilde{m}\right\|_{\tilde X}
= \lim_{\mathcal{U}} \left\|\int_{A_n} f_n\,dm\right\|_X
\\ &\ge& \frac{1-\varepsilon}{2} \lim_{\mathcal{U}} \| f_n\|_{\lm}
=  \frac{1-\varepsilon}{2}\| \tilde f\|_{\mathcal{U}} .
\end{eqnarray*}
Thus, the norm of $(\lm)\,\tilde{}\,$ and the norm of $L^1(\tilde{m})$
are equivalent on $\lm\,\tilde{}\;$. Hence, $(\lm)\,\tilde{}\,$ is order isomorphic to a
subspace of $L^1(\tilde{m})$ which completes the proof.
\end{proof}

Well known examples of B.f.s.' satisfying the subsequence splitting property include those  Orlicz spaces
satisfying the $\Delta_2$ condition, $q$--concave B.f.s.' for $q<\infty$, and
r.i.\ spaces not containing $c_0$, \cite{weis}.

%%%%%%%%%%%%%%%%%%%%%%%%%%%%%%%%%%%%%%%%%%%%%%%%%%%%%%%%%%%%%%%%%%

\section{The weak Banach-Saks property for $\lm$}\label{S-wbsp}

%%%%%%%%%%%%%%%%%%%%%%%%%%%%%%%%%%%%%%%%%%%%%%%%%%%%%%%%%%%%%%%%%%%%

In the following result we require the vector measure
$m\colon \Sigma\to X$ to be \textit{separable}. In analogy to the
scalar case, this means that the associated pseudometric space $(\Sigma,d_m)$ is separable, that is,
it contains a countable dense subset.
The pseudometric $d_m$ is given by
$$
d_m(A,B):=\|m\|(A\triangle B),\quad A,B\in\Sigma,
$$
where $\|m\|(\cdot)$ is the semivariation of $m$.
For $\eta$ a Rybakov control measure for $m$ (see the Preliminaries),
due to the mutual absolute continuity between  $\eta(\cdot)$ and $\|m\|(\cdot)$,
this it is equivalent to the pseudometric space
$(\Sigma, d_\eta)$ being separable, where  $d_\eta(A,B):=\eta(A\triangle B)$, for $A,B\in\Sigma$.
We point out that $m$ is separable precisely when the B.f.s.\ $\lm$ is separable, \cite{ricker2}.

%%%%%%%%%%%%%%%%%%%%%%%%%%%%%%%%%%%%%%%%%%%%%%%%%%%%%%%%%%%%%%%%%%%%

\begin{thm}\label{theorem5}
Let $X$ be  a B.f.s.\ and $m\colon\Sigma\to X$ be a $\sigma$-additive vector measure.
The following conditions are assumed to hold.

\begin{itemize}
\item[(a)] $X$  has the weak Banach-Saks property.
\item[(b)] The measure $m$ is separable  and positive, i.e.,  $m(A)\in X^+$ for $A\in\Sigma$.
\item[(c)] $\lm$ has the subsequence splitting property.
\end{itemize}
Then, the B.f.s.\ $\lm$ has the weak Banach-Saks property.
\end{thm}

\begin{proof}
We need to verify, for a given  weakly null sequence $\{f_n\}\subseteq\lm$, that  there exists
a subsequence $\{f'''_n\}\subseteq \{f_n\}$ whose arithmetic means converge to zero in the norm of $\lm$,
that is,
$$
\lim_{n\to\infty}\Big\|\frac{1}{n}\sum_{k=1}^n f'''_k\Big\|_{\lm} =0.
$$

The proof will be carried out in several steps.

\medskip

\textit{Step 1.} An important observation, which  Szlenk credits to Pe\l czy\'nski, \cite[Remarque 1]{szlenk},
is that the weak Banach-Saks property for a Banach space $Z$ is equivalent to the following
(a priori  stronger)  property:
for every weakly null sequence $\{z_n\}\subseteq Z$ there exists a subsequence
$\{z'_{n}\}\subseteq \{z_n\}$ satisfying
\begin{equation}\label{wBS-2}
\lim_{m\to\infty} \;\sup_{n_1<n_2<\cdots<n_m}\Big\|\frac{1}{m}\sum_{k=1}^m z'_{n_{k}}\Big\|_{Z} =0.
\end{equation}

It is to be remarked that this condition
is a  technical improvement: any further subsequence extracted from $\{z'_{n}\}$
again satisfies \eqref{wBS-2}, for that new subsequence.

\medskip

\textit{Step 2.} Let  $f_n\to0$ weakly in $\lm$. Then,
$\{f_n\}$ is a bounded sequence in $\lm$. Since  $\lm$ has the subsequence splitting
property, there is a subsequence $\{f'_n\}\subseteq \{f_n\}$ and
sequences $\{g_n\}$ and $\{h_n\}$ in $\lm$ such that
\begin{itemize}
\item[(a)] $f'_{n}=g_n+h_n$, with $g_n$ and $h_n$ having disjoint support, $n\in\N$.
\item[(b)]  $\{g_n\}$ has uniformly a.c.\ norm in $\lm$.
\item[(c)] $\{h_n\}$ have pairwise disjoint support.
\end{itemize}

Since $f_n\to0$ weakly in $\lm$, also $f'_n\to0$ weakly in $\lm$. The
claim is that (a), (b), (c) imply that both  $g_n\to0$ weakly in $\lm$ and $h_n\to0$ weakly in $\lm$.

To establish this claim, recall that sets of functions having uniformly a.c.\ norm are
relatively weakly compact (see the Preliminaries).
Thus, from  (b),  the set $\{g_n:n\in\N\}$ is a relatively weakly compact set in $\lm$.
By the Eberlein-\u{S}mulian Theorem, there is a subsequence
$\{g_{n_k}\}$ and $g\in\lm$ such that $g_{n_k}\to g$ weakly in $\lm$.
Since $f_{n_k}\to 0$ weakly in $\lm$, it follows that $h_{n_k}\to (-g)$ weakly in $\lm$.
Let $D_k$ denote the support of $h_{n_k}$; from (c) the sets $D_k$, $k\in\N$, are pairwise disjoint.
Set $E:=\cup_1^\infty D_k$ and  $E_j:=\cup_{1}^{j}D_k$.
Since $L^\infty\subseteq\lm^*$, we have $\chi_A\in\lm^*$ for every $A\in\Sigma$ .
Let $A\in\Sigma$ with $A\subseteq E^c$. Then, $
\langle\chi_A,h_{n_k}\rangle \to \langle\chi_A,(-g)\rangle$. But,
$\langle\chi_A,h_{n_k}\rangle=0$ for all $k\ge1$ and so $g=0$ a.e.\ on
$E^c$. Fix $j\in\N$. For any $A\in\Sigma$ with $A\subseteq E_j$ we have
$\langle\chi_A,h_{n_k}\rangle \to \langle\chi_A,(-g)\rangle$. But,
$\langle\chi_A,h_{n_k}\rangle=0$ for all $k>j$ and so $g=0$ a.e.\ on
$E_j$. Since this occurs for all $j\in\N$, it follows that $g=0$ a.e.\ on $E$. Consequently,
$g=0$ a.e.\ and so $g_{n_k}\to0$ weakly.
This argument shows that the sequence  $\{g_{n}\}$ has the property that, for each of its subsequences,
there is a further subsequence which converges weakly to zero.
This implies that the original sequence $g_{n}\to0$ weakly. Consequently, also $h_{n}\to0$ weakly.

\medskip

\textit{Step 3.} Consider the functions $\{h_n\}\subseteq \lm$ from Step 2. They have pairwise disjoint support.
Let $B_n$ be the support of $h_n$, for $n\in\N$, and $B$ be the complement of $\cup_n B_n$. Define
$$
F:= \chi_B+\sum_{n=1}^\infty \textrm{sign} (h_n)\chi_{B_n} ,
$$
where $\textrm{sign} (h_n)=h_n/|h_n|$ on $B_n$. The function $F$ is
measurable and satisfies $|F|\equiv1$. The operator $f\in\lm\mapsto f F\in\lm$ of multiplication by $F$ is a  linear isometric isomorphism
on $\lm$. Since $h_n\to0$ weakly in $\lm$ and $h_n F=|h_n|$, for $n\in\N$, it follows that $|h_n|\to0$ weakly in $\lm$.

Due to the continuity of the integration operator, it follows that $\int_\Omega |h_n|\,dm\to0$ weakly in $X$.
Since $X$ has the weak Banach-Saks property, there exists a subsequence
$\{h'_n\}\subseteq\{h_n\}$ such that
\begin{equation}\label{disjoint}
\lim_{n\to\infty}\Big\|\frac{1}{n}\sum_{k=1}^n \int_\Omega |h'_k|\,dm\Big\|_{X} =0.
\end{equation}
Due to the fact that the vector measure $m$ is positive we have
$$
\|f\|_{\lm}=\|\,|f|\,\|_{\lm}=\left\|\int_\Omega|f|\,dm\right\|_X,\quad  f\in\lm,
$$
\cite[Theorem 3.13]{okada-ricker-sanchez}. This, together with the fact
(due to the supports of the functions $h'_n$, $n\in\N$, being disjoint) that
$\sum_{k=1}^n |h'_k|=|\sum_{k=1}^n h'_k|$ for $n\in\N$ implies,  from \eqref{disjoint}, that
\begin{equation}\label{disjoint2}
\lim_{n\to\infty}\Big\|\frac{1}{n}\sum_{k=1}^n  h'_k\,dm\Big\|_{\lm}=
\lim_{n\to\infty}\Big\|\int_\Omega \bigg(\frac{1}{n}\sum_{k=1}^n |h'_k|\bigg)\,dm\Big\|_{X}=0.
\end{equation}

Note, in view of Step 1, that the above conclusion still holds if we replace
$\{h'_n\}$ by any subsequence $\{h''_n\}\subset\{h'_n\}$.

\medskip

\textit{Step 4.} Consider now the functions $\{g_n\}\subseteq \lm$ from Step 2.
Let $\{g'_{n}\}$ be the subsequence of $\{g_n\}$ corresponding to the
subsequence $\{h'_{n}\}\subseteq\{h_n\}$ from Step 3. Since
$g_n\to0$ weakly in $\lm$, also $g'_n\to0$ weakly in $\lm$

Let $\eta$ be a Rybakov control measure for $m$.
Since $\lm\subseteq L^1(\eta)$ continuously and $g'_n\to0$ weakly in $\lm$, we have that
$g'_n\to0$ weakly in $L^1(\eta)$. Due to  the well known  Koml\'os theorem, \cite[Theorem 1a]{komlos}, applied in $L^1(\eta)$
to the norm bounded sequence $\{g'_n\}$,  there exists a subsequence
$\{g''_n\}\subseteq\{g'_n\}$ and a function $g_0\in L^1(\eta)$ such that, for every further
subsequence $\{g'''_{n}\}\subseteq\{g''_n\}$, we have
\begin{equation*}
\lim_{n\to\infty} \frac{1}{n}\sum_{k=1}^n g'''_{k}(x) \to g_0(x),\quad \eta-a.e.
\end{equation*}

Since $g'_n\to 0$ weakly in $L^1(\eta)$, also $g''_n\to 0$ weakly in $L^1(\eta)$ and so its averages
$\frac{1}{n}\sum_{k=1}^n g''_{k}\to 0$ weakly in $L^1(\eta)$. Set $\xi_n:=\frac{1}{n}\sum_{k=1}^n g''_{k}\in L^1(\eta)$.
Then $\xi_n\to0$ weakly in $L^1(\eta)$ and $\xi_n\to g_0$ $\eta$-a.e. Combining the Egorov theorem and  the Dunford-Pettis
criterion for relative weak compactness in $L^1(\eta)$, \cite[Theorem 5.2.9]{albiac-kalton},
we deduce that $\xi_n\to g_0$ for the norm in $L^1(\eta)$ and so $g_0=0$.

Consequently, we have selected a subsequence $\{g''_n\}\subseteq\{g'_n\}$ with the property that, for  every
subsequence $\{g'''_{n}\}\subseteq\{g''_n\}$, we have
\begin{equation}\label{komlos}
\lim_{n\to\infty} \frac{1}{n}\sum_{k=1}^n g'''_{k}(x) \to 0,\quad m-a.e.
\end{equation}

\medskip

\textit{Step 5.} Due to the separability of the measure $m$, there exists
a sequence $\{A_n\}\subset\Sigma$ which is dense in the pseudometric space
$(\Sigma, d_\eta)$.

We start a diagonalization process. For notational convenience, let
$$
I_m(f,A):=\int_A f\,dm,\quad f\in \lm,\quad A\in\Sigma.
$$

Define $g^{(1)}_n:=g''_n$, $n\in\N$, where $\{g''_n\}$ is the sequence  obtained
in Step 4. Since $g^{(1)}_n\to0$ weakly in $\lm$ and the  operator of integration over $A_1$,
namely
$$
f\in\lm\mapsto I_m(f,A_1)=\int_{A_1} f\,dm\in X
$$
is continuous, it follows that $I_m(g^{(1)}_n,A_1)\to 0$ weakly in $X$. But, $X$ has the weak Banach-Saks property and so
there exists a subsequence of $\{I_m(g^{(1)}_n,A_1)\}$ satisfying the condition \eqref{wBS-2} in $X$.
We  denote that subsequence  by $\{I_m(g^{(2)}_n,A_1)\}$. In this way we have also
selected a subsequence $\{g^{(2)}_n\}\subseteq\{g^{(1)}_n\}$.

Next we apply the same procedure to  the subsequence $\{g^{(2)}_n\}$ and the set $A_2$ as follows.
Since $g^{(2)}_n\to0$ weakly in $\lm$ and  the  operator of integration over $A_2$, i.e.,
$$
f\in\lm\mapsto I_m(f,A_2)=\int_{A_2} f\,dm\in X
$$
is continuous, it follows that $I_m(g^{(2)}_n,A_2)\to 0$ weakly in $X$. But, $X$ has the weak Banach-Saks property
and so there exists a subsequence of $\{I_m(g^{(2)}_n,A_2)\}$ satisfying the condition \eqref{wBS-2} in $X$.
We denote that subsequence   by $\{I_m(g^{(3)}_n,A_2)\}$. In this way we have
selected a subsequence $\{g^{(3)}_n\}\subseteq\{g^{(2)}_n\}$.
Note,  from Step 1, that  $\{I_m(g^{(3)}_n,A_1)\}$ also satisfies the condition \eqref{wBS-2} in $X$.

For the general step, consider the subsequence $\{g^{(k)}_n\}\subseteq \{g^{(k-1)}_n\}$.
Since $g^{(k)}_n\to0$ weakly in $\lm$ and  the  operator of integration over $A_k$, i.e.,
$$
f\in\lm\mapsto I_m(f,A_k)=\int_{A_k} f\,dm\in X
$$
is continuous, it follows that $I_m(g^{(k)}_n,A_k)\to 0$ weakly in $X$. But,  $X$ has the weak Banach-Saks property
and so there exists a subsequence of $\{I_m(g^{(k)}_n,A_k)\}$ satisfying the condition \eqref{wBS-2}.
We denote that subsequence by $\{I_m(g^{(k+1)}_n,A_k)\}$. In this way we have also
selected a subsequence $\{g^{(k+1)}_n\}\subseteq\{g^{(k)}_n\}$.
Note,  from Step 1, that also $\{I_m(g^{(k+1)}_n,A_j)\}$ satisfies the condition \eqref{wBS-2} in $X$
for all $1\le j\le k$.

By defining $g'''_n:= g^{(n)}_n$, $ n\in\N$, we obtain a subsequence $\{g'''_n\}\subseteq\{g''_n\}$ satisfying
\begin{equation}\label{wBS-Aj}
\lim_{n\to\infty} \Big\|\frac{1}{n}\sum_{k=1}^n \int_{A_j} g'''_k\,dm\Big\|_{X} =0 ,\quad j=1,2,\dots .
\end{equation}
Set
$$
F_n:=\frac{1}{n}\sum_{k=1}^n  g'''_k,\quad n=1,2,\dots.
$$
Then, $\{F_n\}\subseteq\lm$ and we can write \eqref{wBS-Aj} as
\begin{equation}\label{wBS-Aj2}
\lim_{n\to\infty} \Big\|\int_{A_j} F_n\,dm\Big\|_{X} =0 ,\quad j=1,2,\dots .
\end{equation}

\medskip

\textit{Step 6.}  Since the functions $\{g_n\}$ have uniformly a.c.\ norm in $\lm$,
also  the functions $\{g'''_n\}\subseteq\{g_n\}$ have uniformly a.c.\ norm in $\lm$. Recall that  $\lm$ is
a B.f.s.\ over the finite measure space $(\Omega,\Sigma,\eta)$, where $\eta$ is the
Rybakov control measure in Step 4.
The uniform a.c.\ of the norm of $\{g'''_n\}$  in $\lm$ implies that
for every $\varepsilon>0$ there exists a $\delta>0$ such that
\begin{equation}\label{ac}
\eta(A)<\delta \Rightarrow \sup _n\big\|g'''_n\chi_A\big\|_{\lm}<\varepsilon.
\end{equation}

Our next objective is to extend the validity of $\eqref{wBS-Aj2}$ to an arbitrary measurable set $A\in\Sigma$.
So, fix $A\in\Sigma$ and let $\epsilon>0$. Select $\delta>0$ to satisfy \eqref{ac}.
Due to the separability of $(\Sigma, d_\eta)$ there exists $j\in\N$ such that
$\eta(A\triangle A_j)<\delta$.
Then,
\begin{eqnarray*}
\Big\| \int_{A} F_n\,dm\Big\|_{X}
&\le&
\Big\| \int_{A} F_n\,dm-\int_{A_j} F_n\,dm\Big\|_{X} + \Big\| \int_{A_j} F_n\,dm\Big\|_{X}
\\ &\le&
\frac{1}{n}\sum_{k=1}^n\Big\| \int_{A} g'''_k\,dm-\int_{A_j} g'''_k\,dm\Big\|_{X} + \Big\| \int_{A_j} F_n\,dm\Big\|_{X}
\\ &\le&
\frac{1}{n}\sum_{k=1}^n\big\| g'''_k\chi_{A\triangle A_j} \big\|_{\lm} + \Big\| \int_{A_j} F_n\,dm\Big\|_{X}
\\ &\le&
\varepsilon + \Big\| \int_{A_j} F_n\,dm\Big\|_{X},
\end{eqnarray*}
where we have used $|\chi_{A\setminus A_j}g'''_k|\le |\chi_{A\triangle A_j}g'''_k|$
and $\|\int_\Omega g\,dm\|_X\le \|g\|_{\lm}=\|\,|g|\,\|_{\lm}$ for $g\in\lm$.
Due to \eqref{wBS-Aj2},  the last term can be made smaller than $\varepsilon$ for $n\ge n_0$
and some $n_0\in\N$. Hence,
\begin{equation*}
\lim_{n\to\infty} \Big\|\int_{A} F_n\,dm\Big\|_{X} \to0 ,\quad A\in\Sigma .
\end{equation*}

Note that $\{g'''_n\}\subseteq\{g''_n\}$ implies, from \eqref{komlos}, that $F_n\to0$ a.e.
Consequently, we have a sequence $\{F_n\}$ in $\lm$ such that $F_n\to0$ a.e.\
and $\int_A F_n\,dm\to0$ in $X$, for each $A\in\Sigma$.
These two conditions, via Lemma \ref{lewis}, imply that $F_n\to0$ in $\lm$, that is
\begin{equation}\label{wBS-Aj3}
\lim_{n\to\infty}\Big\|\frac{1}{n}\sum_{k=1}^n g'''_k\Big\|_{\lm} =0.
\end{equation}

\medskip

\textit{Step 7.}  Let $\{h'''_{n}\}$ be the subsequence of $\{h_n\}$ corresponding to the
subsequence $\{g'''_{n}\}$ of $\{g_n\}$ from Step 6. For the subsequence  $f'''_{n}=g'''_n+h'''_n$,
$n\in\N$, of  $\{f_n\}$  it follows from  \eqref{disjoint2} and \eqref{wBS-Aj3} that
\begin{equation*}
\lim_{n\to\infty}\Big\|\frac{1}{n}\sum_{k=1}^n  f'''_k\,dm\Big\|_{\lm}= 0 .
\end{equation*}
This completes the proof.
\end{proof}

The combination of Theorems  \ref{theorem4}  and \ref{theorem5} renders the following result.

\begin{thm}\label{theorem6}
Let $X$ be  a B.f.s.\
and $m\colon\Sigma\to X$ be a  $\sigma$-additive vector measure.
The following conditions are assumed to hold.
\begin{itemize}
\item[(a)] $X$  has the weak Banach-Saks property.
\item[(b)] $X$ and $X^*$ have the subsequence splitting property.
\item[(c)] The measure $m$ is separable  and positive, i.e.,  $m(A)\in X^+$ for $A\in\Sigma$.
\item[(d)] The range  $m(\Sigma)$ of $m$ has  uniformly a.c.\  norm in $X$.
\end{itemize}
Then, the B.f.s.\  $\lm$ has both the subsequence splitting property and the weak Banach-Saks property.
\end{thm}

\medskip

We now turn to the

\begin{proof}[Proof of Theorem \ref{theorem1}]
We will deduce Theorem \ref{theorem1} from Theorem \ref{theorem6}.
We first define the relevant vector measure $m$ and verify that the
conditions (c), (d) of Theorem \ref{theorem6} are satisfied.
So, let
$$
m\colon A\in\Sigma\mapsto m(A):=T(\chi_A)\in X.
$$
It a well defined, finitely additive  measure (as $T$ is linear) with values in $X^+$ (as $T$ is positive).
For the $\sigma$-additivity of $m$,  let $\{A_n\}\subseteq\Sigma$ be pairwise disjoint sets.
Since $\chi_{\cup_1^nA_k}\uparrow\chi_{\cup_1^\infty A_k}$ and $T$ is positive,
it follows that $T(\chi_{\cup_1^nA_k})\uparrow T(\chi_{\cup_1^\infty A_k})$
in $X$. Since $X$ has a.c.\ norm (as it has the subsequence splitting property),
this implies that $T(\chi_{\cup_1^nA_k})$ converges to
$T(\chi_{\cup_1^\infty A_k})$ in the norm of $X$. Hence,
$\sum_1^nm(A_k)\to \sum_1^\infty m(A_k)$ in the norm of $X$, i.e.,
$m$ is $\sigma$-additive.

Next we verify condition (c) of Theorem \ref{theorem6}. The vector measure $m$ is
absolutely continuous with respect to the underlying
measure $\mu$. Indeed, if $\mu(A)=0$  for some $A\in\Sigma$, then $m(A)=T(\chi_A)=0$
(as $T$ is linear). Actually, for any $B\in\Sigma$ with $B\subseteq A$ we have
$\mu(B)=0$ and so $m(B)=0$. This implies that $\|m\|(A)=0$.  It follows
for any given  $\varepsilon>0$ that there is $\delta>0$ such that $\mu(A)<\delta$ implies $\|m\|(A)<\varepsilon$.
Since  $\mu$ is separable, there exists a countable set $\{A_j\}$ which is dense in $(\Sigma,d_\mu)$. For any $A\in\Sigma$
and $\varepsilon>0$, let $\delta>0$ be chosen as above. The separability of $(\Sigma,d_\mu)$ ensures there is $j\in\N$ such
that $\mu(A\triangle A_j)<\delta$ and so $\|m\|(A\triangle A_j)<\varepsilon$. Thus, $\{A_j\}$ is dense  in $(\Sigma,d_{m})$.
Hence,  $m$ is separable.

In order to verify condition (d)  of Theorem \ref{theorem6} note, for every
$A\in\Sigma$, that $0\le T(\chi_A)\le T(\chi_{[0,1]})$.
Then, for any $B\in\Sigma$, it follows that $0\le \chi_BT(\chi_A)\le \chi_BT(\chi_{[0,1]})$ and so
$\|\chi_BT(\chi_A)\|_X\le \|\chi_BT(\chi_{[0,1]})\|_X$.
Since $X$ has a.c.\ norm,  the function $T(\chi_{[0,1]})$ has a.c.\ norm in $X$.
So,  given $\varepsilon>0$ there is $\delta>0$
such that $\mu(B)<\delta$ implies that $\|\chi_BT(\chi_{[0,1]})\|_X<\varepsilon$. Then also
$\|\chi_BT(\chi_A)\|_X<\varepsilon$ for all $A\in\Sigma$, that is, the set $\{T(\chi_A):A\in\Sigma\}$
has uniformly a.c.\ norm in $X$. From $m(\Sigma)=\{T(\chi_A):A\in\Sigma\}$ it follows that
$m(\Sigma)$ has uniformly a.c.\ norm in $X$.

Theorem \ref{theorem6} now implies that  $\lm$ has both the subsequence splitting property
and the weak Banach-Saks property.
It remains to establish the equality between $\lm$ and the optimal domain $[T,X]$.
This is a general fact for optimal domains of kernel operators on spaces with
a.c.\ norm, \cite[Corollary 3.3]{curbera-ricker1}.
\end{proof}

%%%%%%%%%%%%%%%%%%%%%%%%%%%%%%%%%%%%%%%%%%%%%%%%%%%%%%%%%%%%%%%%%%

\section{Applications}\label{S-examples}

%%%%%%%%%%%%%%%%%%%%%%%%%%%%%%%%%%%%%%%%%%%%%%%%%%%%%%%%%%%%%%%%%%%%

We provide an application of Theorem \ref{theorem6} to function spaces
arising from convolution operators on groups. The proof of Corollary \ref{corollary2}
on functions spaces arising from kernel operators on $[0,1]$ is also presented.

\medskip

Let $G$ be a compact, metrizable, abelian group and $\lambda$ denote normalized Haar measure on $G$.
Let $\nu$ be any positive, finite Borel measure on $G$. Define a  vector measure
$m_\nu^{(p)}\colon\mathcal {B}(G)\to L^p(G)$, for each $1< p<\infty$, by convolution with $\nu$, i.e.,
$$
m_\nu^{(p)}(A):=\chi_A*\nu,\quad A\in\mathcal {B}(G).
$$
Note that the space $L^p(G)$ has a.c.\ norm, possesses the subsequence splitting property and has the
weak Banach-Saks property. Moreover, its dual space  $(L^p(G))^*=L^q(G)$, with $1/p+1/q=1$, also has
the subsequence splitting property. In addition,  the vector measure $m_\nu^{(p)}$ is clearly positive and separable
(as  the $\sigma$-algebra $\mathcal{B}(G)$ of Borel subsets of $G$ is countably generated).
Concerning the range of $m_\nu^{(p)}$
being uniformly a.c.\ in $L^p(G)$, it suffices  for this  range
to be relatively compact in $L^p(G)$ (see the Preliminares).
For $1<p<\infty$, this is the case precisely when
$\nu\in M_0( G)$, i.e., the Fourier-Stieltjes coefficients of $\nu$
vanish at infinity on the dual group  of $G$, \cite[Proposition 7.58]{okada-ricker-sanchez}.
In particular, this is so whenever $\nu\in L^1(G)$, that is, whenever $\nu$ has an integrable density with respect to
$\lambda$, i.e.,  $\nu=f\,d\lambda$ with $f\in L^1(G)$.
So,  Theorem \ref{theorem1}  implies that each of the  spaces
$$
L^1(m_\nu^{(p)})=\big\{f: \nu*|f|\in L^p(G)\big\},\quad \nu\ge0,\quad \nu\in M_0(G),
$$
\cite[pp.350--351]{okada-ricker-sanchez}, has the subsequence splitting property and weak Banach-Saks property.
It should be remarked  in the event that the measure $\nu\not\in L^p(G)$, then
the space $L^1(m_\nu^{(p)})$ described above is  situated strictly between $L^p(G)$
and $L^1(G)$, i.e.,
$$
L^p(G)\subsetneq L^1(m_\nu^{(p)}) \subsetneq L^1(G);
$$
see \cite[Proposition 7.83]{okada-ricker-sanchez} and the discussion following that result.
It is known that always $L^1(G)\subsetneq M_0(G)$, \cite[p.320]{okada-ricker-sanchez}.

%%%%%%%%%%%%%%%%%%%%%%%%%%%%%%%%%%%%%%%%%%%%%%%%

\medskip

We now turn to the

\begin{proof}[Proof of Corollary \ref{corollary2}]
We verify that the conditions of Theorem \ref{theorem1} are satisfied.

The Lebesgue measure space $([0,1],\mathcal{M},\lambda)$ is separable. Moreover,
the operator $T_K$ defined by \eqref{tk} is linear and positive (as  the kernel $K\ge0$). To
verify that $T_K$ maps $\linf$ into $X$ note, for each $f\in\linf$,
that $|T(f)|\le T(|f|)\le \|f\|_\infty\, T(\chi_{[0,1]})$.
As the function $T(\chi_{[0,1]})$ belongs to $X$ by assumption, it follows that $T(f)\in\linf$.
So, $T\colon L^\infty([0,1])\to X$.

In the case when $X$ is reflexive,  neither  $X$ nor $X^*$ can contain a subspace isomorphic to $c_0$.
Accordingly, as both $X$ and $X^*$  are r.i., they  have the
subsequence splitting property, \cite[2.6 Corollary]{weis}.
\end{proof}

Corollary \ref{corollary2}  applies to many different situations, e.g.,
to the following kernels on $[0,1]$.
\begin{itemize}
\item[(i)] The Volterra kernel, $K(x,y):=\chi_{\Delta}(x,y)$ with $\Delta:=\{(x,y)\in[0,1]\times[0,1]: 0\le y\le x\}$.
\item[(ii)] The Riemann-Liouville fractional kernel, $K(x,y):=|x-y|^{\alpha-1}$ for $0<\alpha<1$.
\item[(iii)] The Poisson semigroup kernel, $K(x,y):=\arctan(y/x)$ for $x\not=0$ and $K(0,y)=\pi/2$.
\item[(iv)] The kernel  associated with Sobolev's inequality, $K(x,y):= y^{(1/n)-1} \chi_{[x,1]}(y)$ for $n\ge2$.
\item[(v)] The  Ces\`{a}ro kernel, $K(x,y):=(1/x)\chi_{[0,x]}(y)$.
\end{itemize}

All of these kernels generate positive operators on
$\linf$. The function $T_K(\chi_{[0,1]})$ belongs, in all
cases, to  $\linf$ and hence, to all r.i.\ spaces on $[0,1]$. In particular,
the function belongs to all r.i.\ spaces with a.c.\ norm.

In relation to condition (d) of Theorem \ref{theorem6},
let us comment on the range of the associated vector measures.
In the cases (i)--(iv), the range is, in fact,  relatively compact
in $C([0,1])$ and  hence, also in any r.i.\ space $X$; see
\cite[Example 4.25]{okada-ricker-sanchez} and the references given there.
In the case (v), the range is relatively compact in
any r.i.\ space   $X\not=\linf$,
\cite[Theorem 2.1]{curbera-ricker4}.

%%%%%%%%%%%%%%%%%%%%%%%%%%%%%%%%%%%%%%%%%%%%%%%%%%%%%%%%%%%%%%%%%%

\bibliographystyle{amsplain}

%%%%%%%%%%%%%%%%%%%%%%%%%%%%%%%%%%%%%%%%%%%%%%%%%%%%%%%%%%%%%%%%%%%%%%%

\end{document}